\documentclass[12pt]{amsart}

\usepackage{hyperref}
\usepackage{amssymb}
\usepackage[bbgreekl]{mathbbol}
\usepackage{graphicx}
\usepackage{ifthen}
\usepackage{pict2e}
\usepackage{xargs}
\usepackage{xspace}
\usepackage{xcolor}
\usepackage{pgf,tikz}
\usepackage{pgfplots}
\usepackage{environ}
\usepackage{caption}
\usetikzlibrary{arrows}
\usepackage{enumitem}
\usepackage{xifthen}
\usepackage[a4paper,left=2cm,right=2cm,top=1.5cm,bottom=1.5cm, includefoot]{geometry}
\usepackage{multicol}
\usepackage{multirow}


\newcounter{dummy}
\makeatletter
\newcommand\myitem[1][]{\item[#1]\refstepcounter{dummy}\def\@currentlabel{#1}}
\makeatother

\makeatletter
\newsavebox{\measure@tikzpicture}
\NewEnviron{scaletikzpicturetowidth}[1]{%
	\def\tikz@width{#1}%
	\begin{lrbox}{\measure@tikzpicture}%
		\BODY
	\end{lrbox}%
	\pgfmathparse{#1/\wd\measure@tikzpicture}%
	\BODY
}
\makeatother

\DeclareSymbolFontAlphabet{\mathbb}{AMSb}

\newcommand{\thistheoremname}{}
\newtheorem*{genericthm*}{\thistheoremname}
\newenvironment{namedthm*}[1]
{\renewcommand{\thistheoremname}{#1}%
	\begin{genericthm*}}
	{\end{genericthm*}}

\newcommand{\Bairespace}[1][]{
	\ifthenelse{\equal{#1}{}}{\functions{\N}{\N}}{\functions{#1}{\N}}
}
\newcommand{\bbL}{\mathbb{L}}
\newcommand{\bbX}{\mathbb{X}}

\newcommand{\Cantorspace}[1][]{
	\ifthenelse{\equal{#1}{}}{\functions{\N}{2}}{\functions{#1}{2}}
}

\newcommandx{\concatenation}[2][1 = undefined, 2 = undefined]{
	\ifthenelse{\equal{#1}{undefined}}{{}\smallfrown}{
		\ifthenelse{\equal{#2}{undefined}}{\bigoplus #1}{\bigoplus_{#1} #2}
	}
}

\newcommandx{\functions}[3][3 =]{
	\ifthenelse{\equal{#3}{}}{#2^{#1}}{#2_{#3}^{#1}}
}

\newcommand{\Gzero}[1][]{
	\ifthenelse{\equal{#1}{}}
	{\mathbb{G}_0}
	{\mathbb{G}_{0,n}}
}
\newcommandx{\Hzero}[2][2 = undefined]{
	\ifthenelse{\equal{#2}{undefined}}
	{\mathbb{H}_{#1}}
	{\mathbb{H}_{#1, #2}}
}
\newcommandx{\intersection}[2][1 =, 2 =]{
	\ifthenelse{\equal{#1}{}}{\cap}{
		\ifthenelse{\equal{#2}{}}{\bigcap #1}{{\bigcap_{#1} #2}}
	}
}
\newcommand{\Lzero}[1][]{\ifthenelse{\equal{#1}{}}{\bbL_0}{L_{0, #1}}}
\newcommand{\Lzerospace}[1][]{\ifthenelse{\equal{#1}{}}{\bbX_0}{X_{0, #1}}}

\newcommand{\modulo}[1]{\ (\text{mod } 2)}
\newcommand{\N}{\mathbb{N}}

\newcommandx{\product}[2][1 =, 2 =]{
	\ifthenelse{\equal{#1}{}}{\times}{
		\ifthenelse{\equal{#2}{}}{\prod #1}{{\prod_{#1} #2}}
	}
}

\newcommandx{\sequence}[2][2 = undefined]{
	\ifthenelse{\equal{#2}{undefined}}{(#1)}{
		(#1)_{#2}
	}
}

\newcommandx{\set}[2][2 = undefined]{
	\ifthenelse{\equal{#2}{undefined}}{\{ #1 \}}{
		\{ #1 \suchthat #2 \}
	}
}
\newcommandx{\sets}[3][3 =]{
	\ifthenelse{\equal{#3}{}}{[#2]^{#1}}{[#2]^{#1}_{#3}}
}

\newcommand{\suchthat}{\mid}

\renewcommand{\restriction}[2]{#1 {\upharpoonright} #2}

\newcommandx{\union}[2][1 =, 2 =]{
	\ifthenelse{\equal{#1}{}}{\cup}{
		\ifthenelse{\equal{#2}{}}{\bigcup #1}{{\bigcup_{#1} #2}}
	}
}

\newtheorem{theorem}{Theorem}[section]
\newtheorem{lemma}[theorem]{Lemma}
\newtheorem{sublemma}[theorem]{Sublemma}

\newtheorem{corollary}[theorem]{Corollary}
\newtheorem{proposition}[theorem]{Proposition}

\newtheorem{question}[theorem]{Question}
\newtheorem{problem}[theorem]{Problem}

\newtheorem{definition}[theorem]{Definition}

\newtheorem{remark}[theorem]{Remark}

\numberwithin{equation}{section}

\newcommand{\oom}{\mathbb{N}^\mathbb{N}}

\newcommand{\bd}{\begin{definition}}
	\newcommand{\ed}{\end{definition}}

\DeclareMathOperator{\im}{im}

\DeclareMathOperator{\dist}{dist}
\DeclareMathOperator{\didistance}{didist}
\newcommand{\mc}{\mathcal}
\newcommand{\mb}{\mathbf}
\newcommand{\bs}{\mathbf{\Sigma}^1_1}

\newcommand{\distance}[3]{\ifthenelse{\isempty{#3}}{\dist(#1,#2)}{\dist^{#3}(#1,#2)}}
\newcommand{\didist}[3]{\ifthenelse{\isempty{#3}}{\didistance(#1,#2)}{\didistance^{#3}(#1,#2)}}
\newcommand{\digraph}[3]{\ifthenelse{\equal{#1}{b}}{\mathbb{#2}_{#3}}
	{{#2}_{#3}}}
\newcommand{\linegraph}[3]{\ifthenelse{\equal{#1}{b}}{\mathbb{#2}_{#3}}
	{#2_{#3}}}

\newcommand{\underlyingspace}[3]{\ifthenelse{\equal{#1}{b}}{\mathbb{#2}_{#3}}
	{#2_{#3}}}
\newcommand{\distanceset}[2]{\ifthenelse{\isempty{#2}}{D(#1)}{D^{#2}(#1)}}

\newcommand{\concatt}{%
	\mathbin{\raisebox{1ex}{\scalebox{.7}{$\frown$}}}%
}

\pgfplotsset{soldot/.style={color=blue,only marks,mark=*}}
\usetikzlibrary{math}
\usetikzlibrary{arrows.meta}
\usetikzlibrary{shapes,decorations}
\usetikzlibrary{decorations.pathmorphing}
\usetikzlibrary{math}
\usetikzlibrary{arrows.meta}
\usetikzlibrary{shapes,decorations}
\usetikzlibrary{decorations.pathmorphing}
\usetikzlibrary{calc}
\definecolor{pastelred}{rgb}{1.0, 0.41, 0.38}
\definecolor{pastelblue}{rgb}{0.52, 0.63, 0.94}
\definecolor{pastelyellow}{rgb}{0.99, 0.99, 0.59}
\definecolor{pastelgreen}{rgb}{0.47, 0.87, 0.47}
\definecolor{pastelorange}{rgb}{1.0, 0.7, 0.28}

\usepackage{tikz-cd}

\newcommand{\DELTA}{\mathbf{\Delta}}
\newcommand{\ER}{\mathrm{ER}}
\newcommand{\GAMMA}{\mathbf{\Gamma}}
\newcommand{\LAMBDA}{\mathbf{\Lambda}}
\newcommand{\PI}{\mathbf{\Pi}}
\newcommand{\red}{\mathrm{red}}
\newcommand{\SIGMA}{\mathbf{\Sigma}}

\DeclareMathOperator{\Bor}{Bor}
\DeclareMathOperator{\code}{\mathsf{code}}

\DeclareMathOperator{\id}{id}

\usepackage{framed}
\definecolor{shadecolor}{named}{lightgray}

\newif\ifold
\oldtrue
\oldfalse

\title{Hyper-hyperfiniteness and complexity}

\author{Joshua Frisch}
\author{Forte Shinko}
\address{
  University of California San Diego,
  Department of Mathematics,
  9500 Gilman Drive,
  La Jolla,
  CA 92093
}

\author{Zolt\'an Vidny\'anszky}
\address{
  E\"otv\"os Lor\'and University,
  Institute of Mathematics,
  P\'azm\'any P\'eter stny. 1/C,
  1117 Budapest,
  Hungary
}

\thanks{The authors are grateful for the support provided by the Erd\H os Center, which made this research possible. The third author was supported by the Hungarian Academy of Sciences Momentum
	Grant no. 2022-58 and National Research, Development and Innovation Office (NKFIH) grants
	no. 113047, 129211}
\begin{document}

\begin{abstract}
  We show that if there exists a countable Borel equivalence relation
  which is hyper-hyperfinite but not hyperfinite,
  then the complexity of hyperfinite countable Borel equivalence relations
  is as high as possible, namely, $\SIGMA^1_2$-complete.
  We also establish an implication between
  the question of the effectivity of hyperfiniteness and its complexity. 
\end{abstract}

\maketitle

The theory of countable Borel equivalence relations (CBERs)
is an active field of research,
which became one of the most prominent directions in descriptive set theory;
for a comprehensive survey see \cite{kechris2024theory}.
Recall that a CBER is \emph{hyperfinite}
if it can be expressed as an increasing union of CBERs with finite classes.
Probably the most investigated general goal of this area
concerns the understanding of hyperfinite CBERs.
Even though the notion of ``understanding" is not a well-defined one,
there are several conjectured statements that would provide significant insight.
One such formal statement is the following:
\begin{question}\label{q:pi11}
  Does the collection of hyperfinite CBERs form a $\PI^1_1$ set
  (see Section \ref{s:prel})?
\end{question}
A seemingly unrelated major open problem in the theory of CBERs
is whether hyper-hyperfinite equivalence relations are hyperfinite:

\begin{question}\label{q:inc}
  Assume that $F_0 \subseteq F_1 \subseteq \dots$ is a sequence of hyperfinite CBERs
  and $E = \bigcup_{i \in \N} F_i$.
  Is $E$ hyperfinite?
\end{question}
While these two statements are sometimes considered ``equally hard",
so far, no formal implication has been shown to hold between them.
We give a short proof of the following, somewhat surprising, implication.
\begin{theorem}\label{t:main}
  A negative answer to Question \ref{q:inc} implies
  a negative one to Question \ref{q:pi11}.
  More precisely,
  if there exists a CBER $E$ which is hyper-hyperfinite but not hyperfinite,
  then hyperfinite equivalence relations form a $\SIGMA^1_2$-complete set. 
\end{theorem}

Roughly speaking,
a subset of the reals is $\SIGMA^1_2$-complete
if it can be defined using an existential and a universal quantifier over the reals,
but among such sets it has the maximal complexity.
In particular,
such a complexity result excludes most of the natural ways
of characterizing hyperfinite CBERs.
Thus, already a negative answer to the increasing union conjecture
would imply that describing hyperfinite CBERs is in some sense a necessarily complicated problem. 

It is also a long standing open question,
originating in \cite{doughertyjacksonkechris},
whether,
assuming that a CBER is hyperfinite,
it is possible to witness this fact effectively
(see, e.g., \cite{moschovakis2009descriptive}
for the basics of effective descriptive set theory).
More precisely:
\begin{question}\label{q:eff}
  Assume that $E$ is a hyperfinite $\Delta^1_1$ CBER,
  is there a $\Delta^1_1$ increasing sequence $(F_n)_{n \in \N}$ of CBERs with finite classes
  such that $E = \bigcup_{n \in \N} F_n$?
\end{question}

It seems rather likely (albeit, we do not see a formal argument)
that a negative answer to the above question
would also yield a \emph{uniform negative answer},
namely, relatively to any $p \in \N^\N$,
a hyperfinite $\Delta^1_1(p)$ CBER for which
there is no $\Delta^1_1(p)$ sequence $(F_n)_n$.
We show that this uniform version also implies complexity,
under determinacy assumptions.

\begin{theorem}[$\SIGMA^1_2$-Determinacy]\label{t:main_2}
  A uniform negative answer to Question \ref{q:eff}
  implies a negative answer to Question \ref{q:pi11}.
  More precisely,
  assume that for each $p\in \N^\N$,
  there exists a hyperfinite $\Delta^1_1(p)$ CBER $E$ on $\N^\N$ so that
  there is no $\Delta^1_1(p)$ increasing sequence of CBERs $(F_n)_{n \in \N}$
  with finite classes so that $E = \bigcup_n F_n$.
  Then the class of hyperfinite CBERs is $\SIGMA^1_2$-complete.
\end{theorem}
The proof of this result is based on a characterization of $\SIGMA^1_2$-completeness of families of sets, which is interesting on its own (see Theorem \ref{t:dichotomy}). 
\medskip

\noindent {\sc Roadmap.}
In Section \ref{s:prel} we collect the most important prerequisites,
in Section \ref{s:main} we prove the main result.
This proof is based on a black box on $\SIGMA^1_2$-completeness,
which is proved in Section \ref{s:complexity}.
Then, in Section \ref{s:under} we provide the proof of Theorem \ref{t:main_2}, while we finish the paper with listing some open problems.
Methods and notions of effective descriptive set theory
are confined to Section \ref{s:under}.
\medskip

\noindent {\sc AI Disclosure.} ChatGPT-6 has been used to proofread this manuscript. 
\section{Preliminaries on coding, complexity, and equivalence relations}
\label{s:prel}

Open, closed, Borel, analytic, co-analytic,
and projections of co-analytic sets
are denoted by
$\SIGMA^0_1$,
$\PI^0_1$,
$\DELTA^1_1$,
$\SIGMA^1_1$,
$\PI^1_1$,
and $\SIGMA^1_2$
(see \cite{kechris2024theory} for basic results about these classes).

A \textit{coding} of a set $A$ is
a surjective partial function $\code \rightharpoonup A$
with $\code$ a standard Borel space,
and we say that $c \in \code$ \textit{codes} $a \in A$
if the coding sends $c$ to $a$.
A \textit{coded set} is a set $A$ equipped with a coding,
which we denote by $\code_A \rightharpoonup A$.

Given coded sets $A$ and $B$,
a function $A \to B$ is \textit{Borel} if it has a Borel lift,
where a \textit{lift} is a function $\code_A \to \code_B$
whose composition with $\code_B \rightharpoonup B$
extends the composition $\code_A \rightharpoonup A \to B$.
Note that every constant function between coded sets is a Borel map.

Let $A$ be a coded set.
We view every subset $B \subseteq A$ as a coded set,
where we equip $B$ with the maximal coding $\code_A \rightharpoonup B$
whose composition with $B \hookrightarrow A$
is extended by the coding $\code_A \rightharpoonup A$.

Given a countable family $\mathcal{A}$ of coded sets,
we view their product $\prod \mathcal{A}$ as a coded set,
equipped with the product coding
$\prod_{A \in \mathcal{A}} \code_A \rightharpoonup \prod_{A \in \mathcal{A}} A$.

We view every standard Borel space $X$ as a coded set
where $\code_X := X$ and the coding is the identity function.

We now describe how we view the set of Borel maps as a coded set.
For standard Borel spaces $X, Y, Z$,
we say that a partial function $Z \rightharpoonup \Bor(X, Y)$
is \textit{$\PI^1_1$-measurable}
if its domain is $\PI^1_1$
and its uncurrying $Z \times X \rightharpoonup Y$ is $\PI^1_1$-measurable,
meaning that the preimage of every $\PI^1_1$ set is $\PI^1_1$.
Given standard Borel spaces $X$ and $Y$,
there is a unique
(up to isomorphism\footnotemark)
$\PI^1_1$-measurable coding $\code \rightharpoonup \Bor(X, Y)$
which is ``universal'' in the following sense:
for every standard Borel space $Z$,
every $\PI^1_1$-measurable $Z \rightharpoonup \Bor(X, Y)$
factors through an injective Borel map $Z \hookrightarrow \code$
(meaning that the partial function $Z \rightharpoonup \Bor(X, Y)$
is equal to the composition $Z \hookrightarrow \code \rightharpoonup \Bor(X, Y)$).
For $Y = 2$,
the existence of such a coding
is described in \cite[35.B]{kechrisclassical}.
For $Y = 2^\N$,
the coding on $\Bor(X, 2^\N)$
is obtained from its identification with $\Bor(X, 2)^\N$,
which has the product coding.
For general $Y$,
by fixing a Borel embedding from $Y$ to $2^\N$,
we can view $\Bor(X, Y)$ as a subset of $\Bor(X, 2^\N)$,
and the restricted coding works.
So we will view $\Bor(X, Y)$ as a coded set
with this canonical coding.

We will just need the following properties of the coding:
\footnotetext{
  Two codings $\code$ and $\code'$ of a set $A$ are \textit{isomorphic}
  if there is a Borel isomorphism $\code \to \code'$
  such that the coding $\code \rightharpoonup A$ is equal to
  the composition $\code \to \code' \rightharpoonup A$.
}
\begin{itemize}
  \item
    Given standard Borel spaces $P, X, Y$,
    the set of Borel functions from $P \times X$ to $Y$ is in bijection,
    via currying,
    with the set of Borel functions from $P$ to $\Bor(X, Y)$.
  \item
    Given standard Borel spaces $X, Y, Z$,
    the composition function
    $\Bor(Y, Z) \times \Bor(X, Y) \to \Bor(X, Z)$
    is Borel.
  \item
    Given a standard Borel space $X$
    and a countable family $\mathcal{Y}$ of standard Borel spaces,
    the natural bijections
    \begin{align*}
      \prod_{Y \in \mathcal{Y}} \Bor(X, Y)
      & \to \Bor\left(X, \prod_{Y \in \mathcal{Y}} Y\right) \\
      \prod_{Y \in \mathcal{Y}} \Bor(Y, X)
      & \to \Bor\left(\coprod_{Y \in \mathcal{Y}} Y, X\right)
    \end{align*}
    and their inverses are Borel.
\end{itemize}

We view $\DELTA^1_1(X)$ as a coded set
via its identification with $\Bor(X, 2)$.

Let $\GAMMA$ be a pointclass of subsets of standard Borel spaces,
and let $A$ be a coded set.
We say that $A$ is $\GAMMA$ if the domain of its coding is $\GAMMA$.
A subset $B \subseteq A$ is \emph{$\GAMMA$-hard}
if for every standard Borel space $X$
and every $C \in \GAMMA(X)$,
there is a Borel function $X \to A$ such that the preimage of $B$ is $C$.

Note that if $\GAMMA$ is a pointclass of subsets of standard Borel spaces
which is closed under finite intersection and Borel preimages,
and $A$ and $B$ are coded sets where $A$ is $\GAMMA$,
then for every Borel map $A \to B$, the preimage of every $\GAMMA$ subset of $B$ is $\GAMMA$.

Let $\mathcal{L}$ be a (relational) language,
by which we mean a family
$(\mathcal{L}_n)_{n \le \omega}$ of sets;
the elements of $\mathcal{L}_n$ are called \emph{$n$-ary relation symbols}
(note that we are allowing $\omega$-ary relations).
Given a set $X$,
an \emph{$\mathcal{L}$-structure on $X$} is a subset $\mathcal{G}$ of
$X_{\mathcal{L}} := \coprod_{n \le \omega} \mathcal{L}_n \times X^n$,
and for all $n \le \omega$ and $l \in \mathcal{L}_n$,
we denote by $\mathcal{G}_l$ the unique subset of $X^n$ such that
$(\{l\} \times X^n) \cap \mathcal{G} = \{l\} \times \mathcal{G}_l$.

Given sets $X$ and $Y$ and a function $\phi : X \to Y$,
we denote by $\phi_{\mathcal{L}}$
the induced function $X_{\mathcal{L}} \to Y_{\mathcal{L}}$,
and given $\mathcal{L}$-structures $\mathcal{G}$ and $\mathcal{H}$ on $X$ and $Y$ respectively,
we say that $\phi$ is a
\emph{homomorphism from $\mathcal{G}$ to $\mathcal{H}$}
if $\mathcal{G} \subseteq \phi_{\mathcal{L}}^{-1}(\mathcal{H})$,
and we say that $\phi$ is a
\emph{reduction from $\mathcal{G}$ to $\mathcal{H}$}
if $\mathcal{G} = \phi_{\mathcal{L}}^{-1}(\mathcal{H})$.
Given an $\mathcal{L}$-structure $\mathcal{G}$ on a set $X$
and a subset $A \subseteq X$,
the \emph{restriction of $\mathcal{G}$ to $A$},
denoted $\restriction{\mathcal{G}}{A}$,
is the $\mathcal{L}$-structure on $A$ defined by
$\restriction{\mathcal{G}}{A}
:= \mathcal{G} \cap \coprod_{n \leq \omega} \left(\mathcal{L}_n \times A^n\right)$.

Let $\mathcal{L}$ be a \emph{Borel} language,
that is,
a language with each $\mathcal{L}_n$ a standard Borel space.
Given a standard Borel space $X$,
a $\mathcal{L}$-structure on $X$ is \emph{Borel}
if it is Borel as a subset of $X_{\mathcal{L}}$.
Note that the restriction of a Borel $\mathcal{L}$-structure to a Borel set
is a Borel $\mathcal{L}$-structure.

We have the following calculation:
\begin{proposition}\label{pr:hompi11}
  Let $\mathcal{L}$ be a Borel language,
  and let $X$ and $Y$ be standard Borel spaces.
  Then the coded set
  \[
    \{
      (\mathcal{G}, \mathcal{H}, \phi)
      \in \DELTA^1_1(X_{\mathcal{L}}) \times \DELTA^1_1(Y_{\mathcal{L}}) \times \Bor(X, Y)
      : \textnormal{$\phi$ is a homomorphism from $\mathcal{G}$ to $\mathcal{H}$}
    \}
  \]
  is $\PI^1_1$.
\end{proposition}
\begin{proof}
  The coded set
  \[
    \{(A, B) \in \DELTA^1_1(X_{\mathcal{L}})^2 : A \subseteq B\}
  \]
  is $\PI^1_1$,
  and the function
  \[
    \begin{aligned}
      \DELTA^1_1(X_{\mathcal{L}}) \times \DELTA^1_1(Y_{\mathcal{L}}) \times \Bor(X, Y)
      & \to \DELTA^1_1(X_{\mathcal{L}})^2 \\
      (\mathcal{G}, \mathcal{H}, \phi)
      & \mapsto (\mathcal{G}, \phi_{\mathcal{L}}^{-1}(\mathcal{H}))
    \end{aligned}
  \]
  is Borel with $\PI^1_1$ domain.
  So the set in question is the preimage of a $\PI^1_1$ set
  under a Borel map with $\PI^1_1$ domain,
  and is hence $\PI^1_1$.
\end{proof}

In particular,
we obtain the following complexity bound:
\begin{proposition}\label{pr:homsigma12}
  Let $\mathcal{L}$ be a Borel language,
  let $X$ be a standard Borel space,
  and let $\mathcal{H}$ be a Borel $\mathcal{L}$-structure.
  Then the coded set
  \[
    \{
      \mathcal{G} \in \DELTA^1_1(X_{\mathcal{L}})
      : \textnormal{$\mathcal{G}$ admits a Borel homomorphism to $\mathcal{H}$}
    \}
  \]
  is $\SIGMA^1_2$.
\end{proposition}

\subsection{On reductions}
Let $\mathcal{L}$ be a language.
Denote by $\mathcal{L}^{\red}$ the language with
$\mathcal{L}^{\red}_n := \mathcal{L}_n \times 2$.
Given a $\mathcal{L}$-structure $\mathcal{G}$ on a set $X$,
denote by $\mathcal{G}^{\red}$ the $\mathcal{L}^{\red}$-structure on $X$
such that for all $n \le \omega$ and all $l \in \mathcal{L}_n$,
we have $\mathcal{G}^{\red}_{(l, 0)} = \mathcal{G}_l$
and $\mathcal{G}^{\red}_{(l, 1)} = X^n \setminus \mathcal{G}_l$.
Then for $\mathcal{L}$-structures $\mathcal{G}$ and $\mathcal{H}$
on sets $X$ and $Y$ respectively,
a function $X \to Y$ is a reduction from $\mathcal{G}$ to $\mathcal{H}$
iff it is a homomorphism from $\mathcal{G}^{\red}$ to $\mathcal{H}^{\red}$.
We will use this to obtain results about reductions
from our results about homomorphisms.
For instance,
since the function
$\DELTA^1_1(X_{\mathcal{L}}) \to \DELTA^1_1(X_{\mathcal{L}^{\red}})$
defined by $\mathcal{G} \mapsto \mathcal{G}^\red$
is Borel,
the coded set
\[
  \{
    (\mathcal{G}, \mathcal{H}, \phi)
    \in \DELTA^1_1(X_{\mathcal{L}}) \times \DELTA^1_1(Y_{\mathcal{L}}) \times \Bor(X, Y)
    : \textnormal{$\phi$ is a reduction from $\mathcal{G}$ to $\mathcal{H}$}
  \}
\]
is $\PI^1_1$,
being the preimage of a $\PI^1_1$ set under a Borel map with $\PI^1_1$ domain.
So in particular,
the coded set
\[
  \{
    (\mathcal{G}, \mathcal{H})
    \in \DELTA^1_1(X_{\mathcal{L}}) \times \DELTA^1_1(Y_{\mathcal{L}})
    : \textnormal{there is a Borel reduction from $\mathcal{G}$ to $\mathcal{H}$}
  \}
\]
is $\SIGMA^1_2$.

\subsection{On hyperfiniteness}
Recall that there is a CBER $\mathbb{E}_0$ on $2^\N$
with the following property:
a CBER is hyperfinite
iff it admits a Borel reduction to $\mathbb{E}_0$
(see \cite[Theorem 5.1]{doughertyjacksonkechris}). 
Thus for every standard Borel space $X$,
we have that $F \in \DELTA^1_1(X^2)$ is a hyperfinite CBER
iff it is a CBER which admits a Borel reduction to $\mathbb{E}_0$.
So for instance,
the coded set
\[
  \{
    F \in \DELTA^1_1(X^2)
    : \textnormal{$F$ is a hyperfinite CBER}
  \}
\]
is $\SIGMA^1_2$,
since being a CBER is $\PI^1_1$
and admitting a Borel reduction to $\mathbb{E}_0$ is $\SIGMA^1_2$.

\section{The proof of the main result}\label{s:main}
We will prove a more general version of the following in the next section.
\begin{theorem}\label{t:complexityhf}
  Suppose that there is a non-hyperfinite CBER $E$
  on a standard Borel space $X$,
  a non-$\sigma$-compact Polish space $P$,
  and a Borel set $U \subseteq P \times X$,
  such that
  \begin{enumerate}[label=(\roman*)]
    \item
      $\restriction{({=}_P \times E)}{U}$ is hyperfinite;
    \item
      for every $x \in X$,
      the set $\{p \in P : x \notin U_p\}$ is $K_\sigma$.
  \end{enumerate}
  Then
  \[
    \{
      F \in \DELTA^1_1((\N^\N)^2)
      : \textnormal{$F$ is a hyperfinite CBER}
    \}
  \]
  is a $\SIGMA^1_2$-hard subset of $\DELTA^1_1((\N^\N)^2)$.
\end{theorem}
The hypothesis says that
there is a family of subsets on which $E$ is hyperfinite,
such that every element of $X$ is in ``almost every'' subset in this family,
where ``almost all'' means ``co-$K_\sigma$''.
If instead we took ``almost all'' to mean ``comeager'',
then we are already able to arrange this situation;
the standard proof that every CBER is generically hyperfinite
is just this construction followed by an application
of Kuratowski-Ulam (see \cite[Theorem 12.1]{topics}).
Similarly,
if ``almost all'' is instead taken to mean ``conull''
(for some Borel probability measure),
then Fubini's theorem forces $E$ to be measure-hyperfinite.

We have the following corollary:
\begin{theorem}\label{hyphyp_implies_complexity}
  Suppose that there exists a hyper-hyperfinite CBER
  which is not hyperfinite.
  Then
  \[
    \{
      F \in \DELTA^1_1((\N^\N)^2)
      : \textnormal{$F$ is a hyperfinite CBER}
    \}
  \]
  is a $\SIGMA^1_2$-hard subset of $\DELTA^1_1((\N^\N)^2)$.
\end{theorem}

\begin{proof}[Proof of \ref{hyphyp_implies_complexity}]
  Fix a hyper-hyperfinite non-hyperfinite CBER $E$
  on a standard Borel space $X$.
  Write $E$ as an increasing union $E = \bigcup_{i \in \N} E_i$
  with each $E_i$ hyperfinite.
  Let $P$ be the Polish space $\N^\N$.

  For every $s \in \N^{<\N}$,
  we will define a finite subequivalence relation $F_s$ of $E$ so that
  \begin{itemize}
    \item if $s \prec t$, we have $F_s \subseteq F_t$,
    \item if $|s| = n$,
      then $F_{s\concatt (k)} \subseteq F_{s\concatt (k + 1)}$ for all $k$,
      and $E_{n+1} = \bigcup_k F_{s\concatt (k)}$.
    \end{itemize}
  For this we use the next standard lemma.
  \begin{lemma}
    Assume that $F \subseteq F'$ are CBERs with $F$ finite and $F'$ hyperfinite.
    Then there is an increasing sequence $(F_n)_{n \in \N}$ of finite superequivalence relations of $F$ so that $\bigcup_n F_n = F'$. 
  \end{lemma}
  \begin{proof}
    Let $f$ be a Borel selector for $F$.
    Then $\restriction{F'}{\im(f)}$ is hyperfinite,
    so it is an increasing union
    $\bigcup_n \tilde F_n$ of finite CBERs.
    Then we are done by taking $F_n = (f \times f)^{-1}(\tilde F_n)$.
  \end{proof}
  Let $F_\varnothing$ be the equality relation.
  Now assume that $F_s$ has been already defined,
  and let $n = |s|$.
  Using the lemma and the assumption that $F_s \subseteq E_n \subseteq E_{n+1}$,
  we can find an increasing sequence $(F'_k)_{k \in \N}$
  of finite superequivalence relations of $F_{s}$
  with $\bigcup_k F'_k = E_{n+1}$.
  Let $F_{s \concatt (k)} = F'_k$. 

  For each $p \in P$,
  define $F_p := \bigcup_{n \in \N} F_{\restriction{p}{n}}$.
  Define $U \subseteq P \times X$ by
  \[
    U :=
    \{
      (p, x) \in P \times X
      : [x]_E = [x]_{F_p}
    \}.
  \]
  We check the hypotheses of \ref{t:complexityhf}:
  \begin{enumerate}[label=(\roman*)]
    \item
      $\coprod_{p \in P} F_p$ is hyperfinite,
      since it is the increasing union over $n \in \N$
      of $\coprod_{p \in P} F_{\restriction{p}{n}}$.
      Now $\restriction{(\coprod_{p \in P} F_p)}{U}$
      is equal to $\restriction{({=}_P \times E)}{U}$,
      so the latter is hyperfinite.
    \item
      First note that for every $(x, x') \in E$,
      the set
      $\{p \in P : (x, x') \notin F_p\}$ is $K_\sigma$:
      this is because if $(x, x') \in E_n$,
      then for every $s \in \N^n$,
      the tree $\{t \in \N^{<\N}: (x, x') \notin F_{s\concatt t}\}$ is finitely branching,
      so the set $\{p \in P: s \prec p \text{ and } (x, x') \notin F_p\}$ is compact.
      
      Now for every $x \in X$,
      the set $\{p \in P : x \notin U_p\}$
      is the union over all $x' \in [x]_E$
      of the sets $\{p \in P : (x, x') \notin F_p\}$,
      so it is also $K_\sigma$.
  \end{enumerate}
\end{proof}

\section{A black box on $\SIGMA^1_2$-hardness}\label{s:complexity}
While for the main result of this paper
we don't need a theorem which is as general as the one below,
it still might be useful to state and prove it in its current form,
for future applicability.

Given a language $\mathcal{L}$,
an $\mathcal{L}$-structure $\mathcal{G}$ on a set $X$,
and a set $P$,
denote by $\mathcal{G}^{\amalg P}$
the coproduct $\coprod_{p \in P} \mathcal{G}$,
viewed as a structure on $P \times X$.
Note that if $P$ is a standard Borel space
and $\mathcal{G}$ is a Borel $\mathcal{L}$-structure,
then $\mathcal{G}^{\amalg P}$ is also a Borel $\mathcal{L}$-structure.

\begin{theorem}\label{t:complexity}
  Let $\mathcal{L}$ be a Borel language,
  and let $\mathcal{H}$ be a Borel $\mathcal{L}$-structure.

  Let $\mathcal{G}$ be a Borel $\mathcal{L}$-structure
  with underlying standard Borel space $X$,
  which does not admit a Borel homomorphism to $\mathcal{H}$.
  Suppose that there is a non-$\sigma$-compact Polish space $P$
  and a Borel set $U \subseteq P \times X$
  such that
  \begin{enumerate}[label=(\roman*)]
    \item
      there is a Borel homomorphism from
      $\restriction{\mathcal{G}^{\amalg P}}{U}$ to $\mathcal{H}$;
    \item
      for every $x \in X$,
      the set $\{p \in P : x \notin U_p\}$ is $K_\sigma$.
  \end{enumerate}
  Then
  \[
    \{
      V \in \DELTA^1_1(P \times X)
      : \textnormal{there is a Borel homomorphism from
      $\restriction{\mathcal{G}^{\amalg P}}{V}$ to $\mathcal{H}$}
    \}
  \]
  is a $\SIGMA^1_2$-hard subset
  of $\DELTA^1_1(P \times X)$.

  In particular,
  \[
    \{
      \mathcal{K} \in \DELTA^1_1((\N^\N)_{\mathcal{L}})
      : \textnormal{$\mathcal{K}$ admits a Borel homomorphism to $\mathcal{H}$}
    \}
  \]
  is a $\SIGMA^1_2$-hard subset
  of $\DELTA^1_1((\N^\N)_{\mathcal{L}})$.
\end{theorem}

Note that this directly yields Theorem \ref{t:complexityhf}:
\begin{proof}[Proof of Theorem \ref{t:complexityhf}]
  Let $\mathcal{L}_{\ER}$ be the language with one binary relation.
  Applying \ref{t:complexity} with
  $\mathcal{L} := (\mathcal{L}_{\ER})^{\red}$ and
  $\mathcal{G} := E^{\red}$ and
  $\mathcal{H} := (\mathbb{E}_0)^{\red}$
  gives that
  \[
    \{
      V \in \DELTA^1_1(P \times X)
      : \textnormal{$\restriction{E^{\amalg P}}{V}$ is hyperfinite}
    \}
  \]
  is $\SIGMA^1_2$-hard.
  This subset reduces via the Borel map
  \begin{align*}
    \DELTA^1_1(P \times X) & \to \DELTA^1_1((P \times X)^2) \\
    V & \mapsto (\restriction{E^{\amalg P}}{V}) \cup {=}_{P \times X}
  \end{align*}
  to the subset of hyperfinite CBERs,
  so the latter is also $\SIGMA^1_2$-hard.
\end{proof}

We recall the condition from
\cite{todorvcevic2021complexity}
for showing $\SIGMA^1_2$-hardness
of certain subsets of $\DELTA^1_1(X)$.

Let $\GAMMA$ and $\LAMBDA$ be pointclasses of subsets of Polish spaces,
let $X$ be a Polish space
and let $\mathcal{A} \subseteq \GAMMA(X)$.
We say that $\mathcal{A}$ is \emph{$\LAMBDA$ on $\GAMMA$}
if for every $B \in \GAMMA(\N^\N \times X)$,
the subset $\{p \in \N^\N : B_p \in \mathcal{A}\}$ is $\LAMBDA$.
We say that $\mathcal{A}$ is \emph{$\LAMBDA$-hard on $\GAMMA$}
if for every $C \in \LAMBDA(\N^\N)$,
there is some $B \in \GAMMA(\N^\N \times X)$
such that $C = \{p \in \N^\N : B_p \in \mathcal{A}\}$.

Note that if $X$ is a standard Borel space,
then by currying,
a subset of $\DELTA^1_1(X)$ is $\SIGMA^1_2$-hard on $\DELTA^1_1$
iff it is $\SIGMA^1_2$-hard
with respect to the coding on $\DELTA^1_1(X)$.

Let $X$ and $Y$ be Polish spaces,
and let $\Psi \colon \GAMMA(X) \to \PI^1_1(Y)$ be a function.
We say that $\Psi$ is \emph{$\PI^1_1$ on $\GAMMA$}
if for every $U \in \GAMMA(\N^\N \times X)$,
the subset
$\{(p, y) \in \N^\N \times Y : y \in \Psi(U_p)\}$
is $\PI^1_1$.
Let $P$ be a Polish space.
For every $B \in \GAMMA(P \times X)$,
let $\overline{\Psi}(B) \subseteq P \times Y$
be the set of $(p, y) \in P \times Y$
such that $y \in \Psi(B_p)$.
Then let the \emph{uniform family}
$\mathcal{U}^\Psi \subseteq \GAMMA(P \times X)$
be the set of $B \in \GAMMA(P \times X)$
such that $\overline{\Psi}(B)$ has a full Borel uniformization
(that is, it contains the graph of a Borel function $P \to Y$).
Let $\mathcal{F}^\Psi \subseteq \GAMMA(X)$
be the set of $A \in \GAMMA(X)$
such that $\Psi(A)$ is non-empty.
The family $\mathcal{F}^\Psi$ is said to be
\emph{nicely $\SIGMA^1_1$-hard on $\GAMMA$}
if for every $A \in \bs(\N^\N)$,
there exist $B \in \GAMMA(\N^\N \times X)$
and $D \in \bs(\N^\N \times Y)$
with $D \subseteq \overline{\Psi}(B)$
such that for all $s \in \N^\N$,
we have 
\[
  s \in A
  \iff D_s \neq \emptyset
  \iff \Psi(B_s) \neq \emptyset
  \ ( \ \iff B_s \in \mathcal{F}^\Psi).
\]

\begin{theorem}[\cite{todorvcevic2021complexity}, Theorem 1.6]
  \label{t:maincomplex}
  Let $X, Y$ be uncountable Polish spaces,
  and let $\GAMMA$ be a class of subsets of Polish spaces
  containing $\PI^0_1$ and $\SIGMA^0_1$
  which is closed under continuous preimages,
  finite unions and intersections.
  Suppose that $\Psi \colon \GAMMA(X) \to \PI^1_1(Y)$
  is $\PI^1_1$ on $\GAMMA$
  and that $\mathcal{F}^\Psi$ is nicely $\bs$-hard on $\GAMMA$. 
  Then the family $\mathcal{U}^\Psi \subseteq \GAMMA(\N^\N \times X)$
  is $\SIGMA^1_2$-hard on $\GAMMA$.
\end{theorem}

We now turn to the proof of the main theorem,
where we use this criterion with $\GAMMA$ being $\DELTA^1_1$:

\begin{proof}[Proof of Theorem \ref{t:complexity}]
  Let $Z$ be the underlying standard Borel space of $\mathcal{H}$.

  Define the function
  $\Psi \colon \DELTA^1_1(X) \to \PI^1_1(\code_{\Bor(X, Z)})$
  by
  \begin{equation}\label{e:Psi}
    \Psi(B) :=
    \{
      c \in \code_{\Bor(X, Z)}
      : \textnormal{$c$ codes a homomorphism
      from $\restriction{\mathcal{G}}{B}$ to $\mathcal{H}$}
    \},
  \end{equation}
  by which we really mean that its restriction to $B$
  is a homomorphism.
  
  Note that,
  by currying,
  the subset $\mathcal{U}^\Psi \subseteq \DELTA^1_1(P \times X)$
  consists exactly of those $V \in \DELTA^1_1(P \times X)$
  for which there is a Borel homomorphism from
  $\restriction{\mathcal{G}^{\amalg P}}{V}$ to $\mathcal{H}$.
  So we need to show that $\mathcal{U}^\Psi$ is a
  $\SIGMA^1_2$-hard subset of $\DELTA^1_1(P \times X)$,
  which we will do by verifying the hypotheses of \ref{t:maincomplex}.

  First of all,
  we check that $\Psi$ is $\PI^1_1$ on $\DELTA^1_1$.
  Let $V \in \DELTA^1_1(P \times X)$.
  Then the set
  \[
    \{(p, c) \in P \times \code_{\Bor(X, Z)}
    : \textnormal{$c$ codes a homomorphism
    from $\restriction{\mathcal{G}}{V_p}$ to $\mathcal{H}$}\}
  \]
  is $\PI^1_1$,
  i.e. the coded set
  \[
    \{(p, \phi) \in P \times \Bor(X, Z)
    : \textnormal{$\phi$ is a homomorphism
    from $\restriction{\mathcal{G}}{V_p}$ to $\mathcal{H}$}\}
  \]
  is $\PI^1_1$,
  because it is the preimage of the $\PI^1_1$ coded set
  \[
    \{(\mathcal{G}, \phi)
    \in \DELTA^1_1(X_{\mathcal{L}}) \times \Bor(X, Z)
    : \textnormal{$\phi$ is a homomorphism
    from $\mathcal{G}$ to $\mathcal{H}$}\}
  \]
  under the Borel map
  \begin{align*}
    P \times \Bor(X, Z)
    & \to \DELTA^1_1(X_{\mathcal{L}}) \times \Bor(X, Z) \\
    (p, \phi)
    & \mapsto (\restriction{\mathcal{G}}{V_p}, \phi)
  \end{align*}
  which has $\PI^1_1$ domain.
  So $\Psi$ is $\PI^1_1$ on $\DELTA^1_1$.

  It remains to show that $\mathcal{F}^\Psi$ is
  nicely $\SIGMA^1_1$-hard on $\DELTA^1_1$.
  Let $A \subseteq \N^\N$ be analytic.

  Since $P$ is not $\sigma$-compact,
  there is a Borel set $F \subseteq \N^\N \times P$
  with closed $P$-fibers whose $\N^\N$-projection is $A$.
  Define $B \subseteq \N^\N \times X$ by
  \[
    B = \{(w, x) \in \N^\N \times X
    : \forall p \in P \; [p \in F_w \implies x \in U_p]\}.
  \]
  So for every $w \in \N^\N$,
  we have $B_w = \bigcap_{p \in F_w} U_p$.
  Note that $B$ is Borel:
  we have $(w, x) \notin B$
  iff $\exists p \in P \; [p \in F_w \text{ and } x \notin U_p]$,
  so the complement of $B$ is the projection of a Borel set with $K_\sigma$ fibers,
  hence Borel.

  Fix a Borel function $P \times X \to Z$
  which is a homomorphism from
  $\restriction{\mathcal{G}^{\amalg P}}{U}$ to $\mathcal{H}$.
  The curried function $P \to \Bor(X, Z)$ is Borel,
  so fix a Borel lift $f : P \to \code_{\Bor(X, Z)}$.
  Note that for every $p \in P$,
  we have $f(p) \in \Psi(U_p)$.
  Define $D \subseteq \N^\N \times \code_{\Bor(X, Z)}$
  to be the image of $F$ under
  ${\id}_{\N^\N} \times f \colon \N^\N \times P \to \N^\N \times \code_{\Bor(X, Z)}$.
  
  For every $w \in \N^\N$,
  we have
  \[
    D_w
    = f(F_w)
    = \bigcup_{p \in F_w} \{f(p)\}
    \subseteq \bigcup_{p \in F_w} \Psi(U_p)
    \subseteq \Psi(B_w).
  \]
  Note that $w \in A$ iff $F_w$ is non-empty
  iff $D_w$ (= $f(F_w)$) is non-empty.
  Furthermore,
  if $w \notin A$,
  then $F_w$ is empty,
  so $B_w = X$,
  and hence $\Psi(B_w) = \Psi(X)$ is empty.

  So $\mathcal{F}^\Psi$ is nicely $\SIGMA^1_1$-hard on $\DELTA^1_1$.
  
  Thus we are done by \ref{t:maincomplex}.
\end{proof}

  \section{Under determinacy}
	\label{s:under}
	It is claimed in \cite{todorvcevic2021complexity} that under $\SIGMA^1_2$-Determinacy one can remove the assumption on niceness from Theorem \ref{t:maincomplex}. Let us prove this statement. This will also yield a nice understanding of complexity in this situation, as well as Theorem \ref{t:main_2}.
	
	\subsection{On effective descriptive set theory}

	First we list some theorems of effective descriptive set theory (see \cite{moschovakis2009descriptive} for the most important theorems and definitions), and prove a couple of auxiliary results. 
	
	Let $\leq_T$ denote the Turing reducibility quasi-order\footnote{We will mostly use $\leq_T$ in this paper. This is a choice only for convenience, other reasonable notions of reduction would work just as well.}. For a set $S \subseteq \N^\N$ and real $r$ denote by $S^{\geq r}$ the set $\{y \in S:y \geq_T r\}$. Let $C, C' \subseteq \oom$ and consider the following version of the Wadge-game (see \cite[21.E]{kechrisclassical}):
	$I$ plays a real $(x_0,x_1,\dots)=x$, $II$ plays a real $(y_0,y_1,\dots)=y$, and $II$ wins iff \begin{equation}y \geq_T x \text{ and } (x \in C' \iff  y \in C). \label{game}\end{equation}
	
	\begin{lemma} Let $C,C'\subseteq \oom$ and assume that the game \eqref{game} is determined. 
		\label{l:game_comp}
		Then one of the following holds:
		\begin{enumerate}
			\item there is a continuous reduction $\varphi$ from $C'$ to $C$ so that for some $p \in \N^\N$ we have that for all $x \geq_T p$ we have $\varphi(x) \equiv_T x$.
			\item there is an $r \in \oom$ and a continuous map $\sigma$ on $\{y: y \geq_T r\}$ that reduces $C^{\geq r}$ to $\oom \setminus C'$. 
		\end{enumerate}

	\end{lemma}
	\begin{proof}
	Clearly, a winning strategy of $II$ yields the desired reduction $\varphi$, and for $p$ one can just choose a real coding $\varphi$. So assume $I$ has a winning strategy, $\sigma$. Let $r=r_\sigma$ be a real code for $\sigma$ and observe that if $y \geq_T r_\sigma$ then $\sigma(y) \leq_T y$. Thus, as $\sigma$ is winning for $I$, if $y \geq_T r_\sigma$, it must be the case that $\sigma(y) \in C' \iff y \not \in C$. 
\end{proof}

Now, we can deduce some corollaries. 	
\begin{corollary} ($\mathbf{\Sigma}^1_2$-Determinacy)
	\label{c:cofin}
	Assume that $C$ is a set such that for each $r$ we have $C^{\geq r} \in \mb{\Pi}^1_1 \setminus \mb{\Delta}^1_1$, and $C' \in \mb{\Pi}^1_1$. Then there is a continuous reduction $\varphi$ of $C'$ to $C$ and a $p$ with $\varphi(x) \equiv_T x$ for all $x\geq_T p$. 
\end{corollary}
\begin{proof}
	Indeed, by the assumption $C^{\geq r} \in \mb{\Pi}^1_1 \setminus \mb{\Delta}^1_1$, the second option of Lemma \ref{l:game_comp} cannot hold, as it would be a reduction of a $\mb{\Pi}^1_1 \setminus \mb{\Delta}^1_1$ to a $\mb{\Sigma}^1_1$ set. 
\end{proof}

\begin{corollary}[$\mb{\Sigma}^1_2$-Determinacy]
	\label{c:pointed}
	Let $S$ be a $\mb{\Sigma}^1_2$ set cofinal in the Turing degrees. Then there exists a continuous map $\varphi$ and a real $p$ with $\varphi(\N^\N) \subseteq S$ and so that for all $x\geq_T p $ we have $\varphi(x)\equiv_Tx$.
	
\end{corollary}
\begin{proof}
	Use Lemma \ref{l:game_comp} for $C=S$ and $C'=\oom$. Then the second option cannot hold, as $\N^\N \setminus C'$ is empty.
\end{proof}
\begin{remark}
	Recall that a non-empty perfect set $Q \subseteq \N^\N$ is \emph{pointed} if for all $x \in Q$ we have $r_Q \leq_T x$, where $r_Q$ is a code for the pruned tree representing $Q$. By passing to a perfect set $P \subseteq \N^\N$ so that $\varphi$ is a homeomorphism on $P$, and every element computes $p$ and the code of $P$, the above corollary also shows that every $\leq_T$ cofinal $\SIGMA^1_2$ set contains a pointed perfect set. 
\end{remark}

	If $r$ is a real and $X$ is a recursively presented Polish space, $\Delta^1_1(X;r)$, $\Pi^1_1(X;r)$, etc., will stand for the appropriate (lightface) class of subsets of $X$, relativized to $r$. Most of the time, we will omit $X$ from the notation. We will identify elements of $\N^\N$, $2^\N$, and $\mathcal{P}(\N)$, using the natural $\Delta^1_1$ bijections between these, and refer to an element of these spaces as a \emph{real}. Below $\mathcal{O}^r$ stands for Kleene's $\mathcal{O}$ relative to $r$.
	\begin{theorem}
	\label{t:effective}
	\begin{enumerate}
		\item (\cite[Section 3.7]{chong2015recursion})  For every $r,p \in \oom$ we have \label{f:effd} $r \in \Delta^1_1(\N;p) \iff$ $\{r\} \in \Delta^1_1(\oom;p) \iff r \in L_{\omega^p_1}[p]$.
		
		\item  \label{f:basis}  (\cite[Lemma 2.4.9]{chong2015recursion}) For every $r,p \in \oom$ we have that $\omega^{r}_1<\omega^{r,p}_1$ implies that $\mathcal{O}^r \in \Delta^1_1(r,p)$.
		\item \label{f:basis2} (\cite[Lemma 2.5.4]{chong2015recursion} and see below) Let $r \in \N^\N$ and $A \in \Sigma^1_1(r)$ be a non-empty subset of $\N^\N$. Then
		\begin{enumerate}
			\item if it is countable, then both $A$ and $\N^\N \setminus A$ contain a $\Delta^1_1(r)$ real.
			\item if it is uncountable, then it contains a non-empty perfect set recursive in $\mathcal{O}^r$. 
		\end{enumerate} 
		\item  \label{f:pi1} (\cite[Theorem 4.1.2]{chong2015recursion}) If $p \in \oom$ and $S$ is a non-empty $\Pi^1_1(p)$ set then there exists an $s \in S \cap L_{\omega_1}[p]$.
		\item \label{f:measbasis} (\cite{sacks1969measure}) If $A \in \Pi^1_1(p)$, $\lambda$ is a $\Delta^1_1$ measure on $\N^\N$, and $\lambda(A)>0$, then $A$ contains a real $\Delta^1_1(p)$.  
		\item (Kleene, folklore)
		\label{f:kleene}
		(\cite[4D.3-4]{moschovakis2009descriptive}) Suppose that $C \subseteq \N^\N \times \N^\N$ is a $\Pi^1_1(p)$ set. Then 
		
		\begin{enumerate}
			\item the set $\{x:(\exists y \in \Delta^1_1(x,p))((x,y)\in C)\}$ is $\Pi^1_1(p)$,
			\item $C$ has a full Borel uniformization if and only if there exists a real $q$ so that for every $x$ we have $\Delta^1_1(x,q) \cap C_x \not = \emptyset$.
		\end{enumerate}
	\end{enumerate}	
\end{theorem}

In order to see the first statement in \eqref{f:basis2} note that $\Delta^1_1(r)$ reals form a $\Pi^1_1(r) \setminus \Sigma^1_1(r)$ set. Moreover, by \cite[Lemma 2.5.4]{chong2015recursion} any $\Sigma^1_1(r)$ set with a non $\Delta^1_1(r)$ member is uncountable, hence if $A$ is countable, it must contain a $\Delta^1_1(r)$ real and it cannot contain all $\Delta^1_1(r)$ reals, since then either $A$ would be equal to the collection of $\Delta^1_1(r)$ reals, in which case they would form a $\Sigma^1_1(r)$ set, or $A$ would contain a non $\Delta^1_1(r)$ real, contradicting its countability. For the second statement, note that given a $\Sigma^1_1(r)$ set $A$, there is a recursive in $r$ tree $T$ so that $A$ contains a perfect subset iff $T$ has an infinite branch. Now the statement follows from \cite[Lemma 2.5.1]{chong2015recursion}.
	\subsection{A characterization result}
	Now we show the following, which we believe to provide a thorough understanding of $\mb{\Sigma}^1_2$-completeness.
    Recall the notions from Section \ref{s:complexity}. 
	\begin{theorem}[$\mb{\Sigma}^1_2$-Determinacy] 
				\label{t:dichotomy} 	Let $X,Y$ be uncountable Polish spaces, $\GAMMA$ be a class of subsets of Polish spaces which is closed under continuous preimages and $\Psi:\mb{\Gamma}(X) \to \mb{\Pi}^1_1(Y)$ is $\mb{\Pi}^1_1$ on $\mb{\Gamma}$. Then the following are equivalent:
		
		\begin{enumerate}
			\item \label{c:set} there is a set $B \in \mb{\Gamma}(\N^\N \times X)$ such that for all $s \in \N^\N$ we have $\Psi(B_s) \neq \emptyset$ and $B \not \in \mc{U}^\Psi$, 
			\item \label{c:not} $\mc{U}^\Psi$ is not $\mb{\Pi}^1_1$ on $\mb{\Gamma}$, 
			\item \label{c:max} $\mc{U}^\Psi$ is $\SIGMA^1_2$-hard on $\GAMMA$.
		\end{enumerate}

	\end{theorem}

	 Clearly, \eqref{c:max} $\implies$ \eqref{c:not}, and we prove the equivalence by showing that \eqref{c:set} $\implies$ \eqref{c:max} and then \eqref{c:not} $\implies$ \eqref{c:set}.
	 Using a Borel bijection between $Y$ and $\N^\N$, and composing it with $\Psi$, we may assume that $Y=\N^\N$. 

	We will use the following observation repeatedly.
	\begin{lemma}
		\label{l:obs}
		Let $B \in \GAMMA(\N^\N \times X)$. Then \[B \in \mathcal{U}^\Psi \iff \exists p_0 \in \N^\N \ \forall p\in \N^\N \ \exists t \in \Delta^1_1(p,p_0) \ (t \in \Psi(B_p)).\]
	\end{lemma}
	\begin{proof}
		This is just a consequence of Theorem \ref{t:effective} \eqref{f:kleene}.
	\end{proof}
   	In the proofs below, we will slightly abuse the notation by using $s(0),s(1) \in \N^\N$ for the first and second coordinates of an element of $\N^\N \times \N^\N$.  
	
\eqref{c:set} $\implies$ \eqref{c:max}.	Assume \eqref{c:set} holds. The main lemma below says that we can construct a set, where all the elements of $\Psi(B_{(s,q,t)})$ are sufficiently complicated, but there are ones which are not too complicated.

\begin{lemma}
	\label{l:appropriate}
	Assume that \eqref{c:set} holds. For every $A \subseteq (\N^\N)^3$ analytic there exists a $B \in \mb{\Gamma}((\N^\N)^3 \times X)$ and a $p \in \oom$ such that 
	\begin{enumerate}
		\item for all $(s,q,t)\in (\N^\N)^3$ we have $\Psi(B_{(s,q,t)}) \cap \Delta^1_1(\mc{O}^{s,q,t,p}) \neq \emptyset,$
		\item $\Psi(B_{(s,q,t)}) \cap \Delta^1_1(s,q,t,p) \neq \emptyset$ iff $(s,q,t) \not \in A$. 
	\end{enumerate}
\end{lemma}
We first deduce the desired implication from this lemma. 

Let $S \in \Sigma^1_2$ be $\mb{\Sigma}^1_2$-complete, and fix a $\Pi^1_1$ set $C_0 \subseteq (\N^\N)^2$ projecting onto $S$. Let \[C=\{(s,q,t): \Delta^1_1(s,q,t)\cap (C_0)_s \neq \emptyset\} \text{ and } A=(\N^\N)^3 \setminus C,\] 
note that $A$ is analytic, by Theorem \ref{t:effective} \eqref{f:kleene}. Pick a set $B$ and real $p$ to $A$ with the properties as in Lemma \ref{l:appropriate}. We claim that \[s \in S \iff B_{(s,p)} \in \mathcal{U}^\Psi,\]
showing the complexity result. 

First, assume that $s \not \in S$, and $p_0$ is arbitrary, we prove that $\overline{\Psi}(B_{(s,p)})$ has no $\Delta^1_1(p_0)$ uniformization, showing $B_{(s,p)} \not \in \mc{U}^\Psi$. We have $(C_0)_s=\emptyset$, so $(s,q,t) \not \in C$ for all $(q,t)$, thus, $(s,p,p_0) \in A$. In particular, by the second property of $B$, $\Psi(B_{(s,p,p_0)}) \cap \Delta^1_1(s,p,p_0)=\emptyset$. Hence indeed, $\Psi(B_{(s,p)})$ cannot have a $\Delta^1_1(p_0)$ uniformization.

Conversely, assume that $s \in S$. As $(C_0)_s \neq \emptyset$, and this is a $\Pi^1_1(s)$ set, it has an element in $L[s]$, by Theorem \ref{t:effective} \eqref{f:pi1}. Also, by definition, it has no element which is $\Delta^1_1(s,p,t)$ whenever $t \in A_{(s,p)}$, thus, it follows from Theorem \ref{t:effective} \eqref{f:effd} that $\{\omega^{(s,p,t)}_1:t \in A_{(s,p)}\}$ is bounded below $\omega^{L[s]}_1$. Pick some $r$ with $\omega^{r}_1$ being strictly greater than all these. 

Then we claim that for all $t$ we have $\Psi(B_{(s,p,t)}) \cap \Delta^1_1(s,p,t,r) \neq \emptyset$, hence using Lemma \ref{l:obs}, we will get $B_{(s,p)} \in \mc{U}^\Psi$. Indeed, if $(s,p,t) \in A$, then $\mc{O}^{s,p,t} \in \Delta^1_1(s,p,t,r)$ by $\omega^{s,p,t}_{1}<\omega^{s,p,t,r}_1$ and Theorem \ref{t:effective} \eqref{f:basis}, so by the first condition on $B$ we have $\Psi(B_{(s,p,t)}) \cap \Delta^1_1(s,p,t,r)\neq \emptyset$. If $(s,p,t) \not \in A$ then already $\Psi(B_{(s,p,t)}) \cap \Delta^1_1(s,p,t) \neq \emptyset$ by the second condition. 

Now we prove the key lemma.

\begin{proof}[Proof of Lemma \ref{l:appropriate}]To ease the notation, we construct such a set $B \in \mb{\Gamma}(\N^\N \times X)$ and real $p$ for every analytic $A \subseteq \N^\N$, and the lemma then can be obtained using the usual bijection between $\N^\N$ and $(\N^\N)^3$. 
	
	Let $B'$ be a set as in \eqref{c:set}, and $C'$ defined by 
	\begin{equation}C'=\{(s,t): \Psi(B'_{s(0)}) \cap \Delta^1_1(s,t)=  \emptyset \land  \Psi(B'_{s(0)}) \cap  \Delta^1_1(\mc{O}^{(s,t)})\neq \emptyset\},
	\label{e:defc'}
	\end{equation}
	
	It is not hard to see using Theorem \ref{t:effective} \eqref{f:kleene} and the definition of $\mathcal{O}$ that $C'$ is $\mb{\Sigma}^1_2$.
	\begin{sublemma}
		\label{sl:cof}
		$C'$ is cofinal in $\leq_T$.
	\end{sublemma}
	
	\begin{proof}
		Let $r$ be arbitrary, we construct an element of $C'$ $\leq_T$ above $r$. W.l.o.g., we may assume $\overline{\Psi}(B') \in \Pi^1_1(r) $. We claim that there is an $s$ with $s(1)=r$ and $\Psi(B'_{s(0)}) \cap \Delta^1_1(s)=\emptyset$. The contrary of this would imply that for all $s(0)$ we have $\Psi(B'_{s(0)}) \cap \Delta^1_1(s(0),r) \neq \emptyset$. But this would mean that $\overline{\Psi}(B')$ admits a Borel, in fact, $\Delta^1_1(r)$ uniformization, showing $B' \in \mc{U}^\Psi$ and contradicting \eqref{c:set}. 
		
		Now consider the set $A'_s=\{t:\Psi(B'_{s(0)}) \cap \Delta^1_1(s,t)=\emptyset\}$. This set is non-empty, $\Sigma^1_1(s)$ by Theorem \ref{t:effective} \eqref{f:kleene}, and uncountable by Theorem \ref{t:effective} \eqref{f:basis2}: indeed, if it was countable, its complement would contain $\Delta^1_1(s)$ reals, in particular $\Psi(B'_{s(0)})$ would contain such reals as well, contradicting the choice of $s$. It follows that it contains a perfect subset $P$ recursive in $\mc{O}^{s}$. Then $\{(\mc{O}^s,t):t \in P\}$ is $\leq_T$ cofinal, consequently, for some $t \in P$ we have $\Psi(B'_{s(0)}) \cap \Delta^1_1(\mc{O}^{s,t}) \neq \emptyset$. 
		
		Thus $(s,t) \in C'$, and $s \geq_T r$, as desired. 
	\end{proof}

	\begin{sublemma}
		\label{sl:intermed}
		There exists a set $B^* \in \mb{\Gamma}(\oom \times X)$ and some $p^* \in \N^\N$, such that for all $s \in \N^\N$ we have 
		\begin{enumerate}
			\item if $s \geq_T p^*$, then $\Psi(B^*_s) \cap \Delta^1_1(s)=\emptyset$,
			\item $\Psi(B^*_{s}) \cap \Delta^1_1(\mc{O}^{(p^*,s)}) \neq \emptyset$.
		\end{enumerate} 
	\end{sublemma}
	\begin{proof}

		Consider $C'$ from \eqref{e:defc'}. By Sublemma \ref{sl:cof} and  Corollary \ref{c:pointed} there is a continuous map $h$ from $\N^\N$ to $C'$ and a $p^*$ so that for all $s \geq_T p^*$ we have $s \equiv_T h(s)$. We can also assume that $p^*$ computes a code of $h$, and hence $h(s) \leq_T (p^*,s)$ for all $s$. Define now $B^* \in \mb{\Gamma}(\oom \times X)$ by
		\[(s,x) \in B^* \iff (h(s)(0)(0),x) \in B'.\] 
		Clearly, $B^*$ works by Lemma \ref{l:obs} and the choice of $C'$ and $h$. 
	\end{proof}

	\begin{sublemma}
		\label{sl:complex}
		Let $B^*$ be as above. Let \[C^*=\{(s,t): \Psi(B^*_s) \cap \Delta^1_1(s,t) \neq \emptyset\}.\] Then $C^*$ is $\mb{\Pi}^1_1 \setminus \mb{\Delta}^1_1$. Moreover, the set $C^{*\geq r}$ is also $\mb{\Pi}^1_1 \setminus \mb{\Delta}^1_1$ for all $r \in \oom$. 
	\end{sublemma}
	\begin{proof}
		Towards contradiction, assume that $C^*$ is $\Delta^1_1(p)$ for some $p \in \oom$. Then by Sublemma \ref{sl:intermed} there must be some $p'$ (e.g., let $p' \geq_T p^*,p$), such that $\Psi(B^*_{p'}) \cap \Delta^1_1(p')=\emptyset$ and $C^*_{p'}$ is a set which is $\Delta^1_1(p')$. Observe that the former implies $C^*_{p'} \cap \Delta^1_1(p')=\emptyset$ as well. 
		
		Fix a $\Delta^1_1$ continuous probability measure on $\oom$, $\lambda$. Then $\oom \setminus C^*_{p'}$ has $\lambda$ measure $1$, since otherwise $C^*_{p'}$ would contain a $\Delta^1_1(p')$ real, by Theorem \ref{t:effective} \eqref{f:measbasis}. Now, pick an arbitrary $t \in C^*_{p'}$. Then, since $\oom \setminus C^*_{p'}$ has measure $1$ it contains a $t'$ with $t \in \Delta^1_1(p',t')$: this can be seen by showing that a measure $1$ $\Delta^1_1(p')$ set contains a $\Delta^1_1(p')$ homeomorphic copy of $2^\N$, using the compact regularity argument. But then $t' \in C^*_{p'}$, contradicting the choice of $t'$. 
		
		The ``moreover" part follows from a relativization of the above argument to $r$. 
	\end{proof}

	\begin{remark}
		An alternative argument for Sublemma \ref{sl:complex} can be given using the Martin-Friedman theorem (\cite[Theorem 11.1.1]{chong2015recursion}, \cite{martin1976proof}), namely that every uncountable $\Delta^1_1$ set contains an element of every large enough hyperdegree. Then the set $\N^\N \setminus C^*_p$ cannot be countable, otherwise its complement would contain a $\Delta^1_1(p)$ real; but then it contains every large enough degree relatively to $p$, contradicting its definition. 
	\end{remark}
	
	In order to finish the proof of Lemma \ref{l:appropriate}, pick an arbitrary analytic $A \subseteq \N^\N$, and define $A' \subseteq \N^\N$ by $s \in A' \iff s(1) \in A$ (i.e., take $\oom$-many copies of $A$).  Now by Corollary \ref{c:cofin} applied to $C^*$ and Sublemma \ref{sl:complex} we obtain a $p_0 \in \oom$ and a continuous reduction $\varphi$ from $\N^\N \setminus A'$ to $C^{*}$ so that $\varphi(s)\equiv_T s$ for all $s \geq_T p_0$. Let $p$ code $\varphi$, $p_0$, and $p^*$ from Sublemma \ref{sl:intermed}. Then, restricting $\varphi$ to $\{s:s(0)=p\}$, that is, considering the map $\varphi'(s)=\varphi(p,s)$ yields a reduction of $\oom \setminus A$ to $C^*$ so that $\varphi'(s) \equiv_T (s,p)$. 
	
	Define \[(s,x) \in B \iff (\varphi'(s)(0),x)\in B^*.\]
	We check that $B$ and $p$ work. 
	
	If $s \in \N^\N$ we have $\Psi(B_{s}) \cap \Delta^1_1(\mc{O}^{s,p}) \neq \emptyset,$ because of Sublemma \ref{sl:intermed} and the fact that $\varphi'(s) \equiv_T (s,p)$. 
	
	Finally, \[s\not \in A \iff \varphi'(s) \in C^* \iff \Psi(B_{s}) \cap \Delta^1_1(s,p) \neq \emptyset,\] using $\varphi'(s) \equiv_T (s,p)$ again.	
\end{proof}

\subsection{\eqref{c:not} $\implies$ \eqref{c:set}}
This part is done similarly to Lemma 3.1 in \cite{todorvcevic2021complexity}. 

Assume that $\mc{U}^\Psi$ is not $\mb{\Pi}^1_1$ on $\mb{\Gamma}$. 

Take a universal set $U\in \bs(\N^\N \times \N^\N \times    \N^\N)$ so that for any $R \in \PI^1_1(\N^\N \times \N^\N)$ there exists an $r \in \N^\N$ with $U_{r}=(\N^\N \times \N^\N) \setminus R$. Define $(r,s) \in A \iff (r,r,s) \in U$. Then clearly $A \in \bs( \N^\N \times \N^\N)$ and if $R \in \PI^1_1(\N^\N \times \N^\N)$ is arbitrary then for some $r \in \N^\N$ we have $(\N^\N \times \N^\N) \setminus R=U_{r}$, so by definition $A_{r}=U_{(r,r)}=\oom \setminus R_r$.

Using the facts that $\mc{U}^\Psi$ is not $\mb{\Pi}^1_1$ on $\mb{\Gamma}$ and that $\mb{\Sigma}^1_2$-Determinacy holds, and observing that having a Borel uniformization is a $\mb{\Sigma}^1_2$ property and Wadge's lemma, we obtain a set $B^0 \in \GAMMA(\N^\N \times \N^\N \times \N^\N \times X)$ such that for every $(r,s)$ we have
\begin{equation}
\label{e:basic}
(r,s) \in A \iff B^0_{(r,s)} \in \mc{U}^\Psi.\end{equation}

Define			
\[R=\{(r,s) \in \N^\N \times \N^\N:\] \[\text{ $
	\overline{\Psi}(B^0_{(r,s)})$ has a $\Delta^1_1(s)$ uniformization}\}.\]

By Theorem \ref{t:effective} \eqref{f:kleene}, we have $R \in \PI^1_1(\N^\N \times \N^\N)$. Then by the choice of $A$ there exists an $r_0 \in \N^\N$ so that $A_{r_0}=\oom \setminus R_{r_0}$. Let \[(t,x) \in B \iff (t(0),x) \in B^0_{(r_0,t(1))}.\]

We claim that $B$ has \eqref{c:set}. It is clear by the assumptions on $\mb{\Gamma}$ that $B$ is in $\mb{\Gamma}$.

We show first that $B \not \in \mc{U}^\Psi$. Indeed, if $\overline{\Psi}(B)$ had a $\Delta^1_1(s)$ uniformization with $s \geq_T r_0$, then in particular $\Psi(B^0_{(r_0,s)})$ had one, so $(r_0,s) \in R$, but then $B^0_{(r_0,s)} \not \in \mc{U}^\Psi$ by \eqref{e:basic}.

Now, to see that for each $t$ the set $\Psi(B_t)$ is non-empty, observe that for any $t$, the set $\overline{\Psi}(B^0_{(r_0,t(1))})$ has a Borel uniformization: indeed, if $(r_0,t(1)) \in R$, then $\Psi(B^0_{(r_0,t(1))})$ has a $\Delta^1_1(t(1))$ uniformization and if $(r_0,t(1)) \not \in R$, then $B^0_{(r_0,t(1))} \in \mc{U}^\Psi$, by the choice of $A$ and $B^0$. In particular, $\Psi(B_t)=\Psi((B^0_{(r_0,t(1))})_{t(0)})$ is not empty. This concludes the proof of Theorem \ref{t:dichotomy}.

	\subsection{The second main result}	Let us deduce finally Theorem \ref{t:main_2}:

	\begin{namedthm*}{Theorem \ref{t:main_2}}		
		\label{t:main2} ($\SIGMA^1_2$-Determinacy) Assume that for each $p\in \N^\N$ there exists a hyperfinite $\Delta^1_1(p)$ CBER $E$ on $\N^\N$, so that there is no $\Delta^1_1(p)$ increasing sequence of CBERs $(F_n)_{n \in \N}$ with finite classes so that $E=\bigcup_n F_n$. Then hyperfinite equivalence relations form a $\SIGMA^1_2$-complete set.
	\end{namedthm*}
	
	\begin{proof}
		We will use \eqref{c:set} from Theorem \ref{t:dichotomy}. Define $\Psi$ to be a map that to a structure $\mathbb{E}^{\red}$ associates codes of Borel homomorphisms from $\mathbb{E}^{\red}$ to $\mathbb{E}^{\red}_0$, similarly to \eqref{e:Psi}. 
		Fix $\Pi^1_1$ sets $U \subseteq \N \times \N^\N \times \N^\N$ (technically, the last coordinate is identified with $(\N^\N)^2 \sqcup (\N^\N)^2$ to encode $\mathbb{E}^{\red}$) and $C\subseteq \N \times \N^\N$ with the following properties:
		
		\begin{itemize}
			\item for all $p$ and $E \in \Delta^1_1(p)$ there is some $n$ with $U_{(n,p)}=E^{\red}$ and $(n,p) \in C$,
			\item for all $B \subseteq C$ Borel, we have that $\{(n,p,x) \in U:(n,p) \in B\}$ is Borel, 
		\end{itemize} 
		It is not hard to see that the usual effective universal $\Pi^1_1$ sets $U$ for Borel sets of $\N^\N$, together with the collection of Borel codes $C$ have these properties (e.g., \cite[3H.1]{moschovakis2009descriptive}). 
		Now consider the set
		\[C^*=\{(n,p) \in C: U_{(n,p)} \text{ codes a hyperfinite CBER}, \Psi(U_{(n,p)}) \cap \Delta^1_1(p)=\emptyset\}.\]
		
		Then the set $C^*$ is $\SIGMA^1_2$, and by the assumption of the theorem, since we can effectively go from finite exhaustion to homomorphism to $\mathbb{E}_0^{\red}$ and vice-versa, for all $p$ there is an $n$ with $(n,p) \in C^*$. Now, by Corollary \ref{c:pointed}, there exists a continuous map $\varphi: \N^\N \to \N \times \N^\N$ and some $p_0 \in \N^\N$, so that $\varphi(s) \equiv_T s$, for all $s \geq_T p_0$ and $\varphi(\N^\N) \subseteq C^*$. Let 
		\[B=\{(s,x):(\varphi(s),x) \in U, s \geq_T p_0\} \cup \bigcup_{s \not \geq_T p_0}\{s\} \times (=_{\N^\N})^{\red}.\]
		By  $\varphi(s) \equiv_T s,$ the map $\varphi$ must be countable-to-$1$ on the set $\{s:s \geq_T p_0\}$, so by the Luzin-Novikov theorem its image is Borel. In particular, by the second assumption on $U$, $B$ is Borel. 
		Then $\Psi(B_{s})=\Psi(U_{\varphi(s)}) \neq \emptyset$ and $\Psi(B_{s}) \cap \Delta^1_1(\varphi(s))=\emptyset$, which implies $\Psi(B_{s}) \cap \Delta^1_1(s)=\emptyset$ whenever $s \geq_T p_0$, and clearly $\Psi(B_{s}) \neq \emptyset$ for every other $s$. It follows from Lemma \ref{l:obs} that $B \not \in \mathcal{U}^{\Psi}$, showing \eqref{c:set}.	
	\end{proof}

	\begin{remark} Assuming additional properties about $\Gamma$ and $\Psi$, for example, if $\mb{\Gamma}=\mb{\Delta}^1_1$, $X=Y=\N^\N$, and $\Psi$ is additionally $\mb{\Pi}^1_1$ on $\mb{\Pi}^1_1$, one can add an equivalent condition to Theorem \ref{t:dichotomy} similar to the assumption of Theorem \ref{t:main_2}:
		\begin{enumerate}
			\item[(4)]  $\forall p \ \exists q \ \exists B \in \Delta^1_1(q) \ (p \in \Delta^1_1(q) \land \Psi(B) \cap \Delta^1_1(q)=\emptyset \land \Psi(B) \neq \emptyset)$,
		\end{enumerate}
		that is, cofinally often there is a counterexample to effectively choosing an element from $\Psi(B)$. 
	\end{remark}

    Finally, let us note that Theorem \ref{t:dichotomy} also directly implies the following theorem, using $\Psi$ defined as above. 
    \begin{corollary} ($\SIGMA^1_2$-Determinacy)
    \label{t:intro3}
    Assume that hyperfinite CBERs do not form a $\mb{\Pi}^1_1$ set. Then they form a $\SIGMA^1_2$-complete set. 
    \end{corollary}
	
	\section{Further problems}
	
	There are several further open problems in the theory of CBERs, which are sometimes considered ``equally hard" (see \cite[Section 17.4]{kechris2024theory}). It would be interesting to understand their relationship to each other and complexity. A possible exciting direction would be to explore the structure of classes of CBERs, for which non-hyperfiniteness implies complexity.
	
	\begin{problem}
		Let $\mathcal{F}$ be a family of CBERs. Assume that there is an $E \in \mathcal{F}$ that is not hyperfinite. Does it follow that hyperfinite CBERs form a $\SIGMA^1_2$-complete set?
		
		In particular, if $\mathcal{F}$ is the collection of
		\begin{enumerate}
			\item measure hyperfinite CBERs,
			\item Borel amenable CBERs,
			\item CBERs on $\N^\N$ that are hyperfinite restricted to every compact set,
			\item Borel bounded CBERs (see \cite{boykin2007borel}),
			\item $\mathbb{Z}^2$-orderable CBERs,
			\item CBERs arising from a Borel action of Thompson's group?
			
		\end{enumerate}
	\end{problem}
	The following converse to our main theorem also seems to be conceivable. 
	\begin{problem}
		Assume that hyper-hyperfinite CBERs are hyperfinite. Do hyperfinite CBERs form a $\mathbf{\Pi}^1_1$ set?
	\end{problem}
	
	A possible way of summarizing Theorem \ref{t:complexityhf} is that if one manages to come up with a CBER on $[\N]^\N$ that is not hyperfinite, but hyperfinite on every non-dominating set (uniformly), then the complexity result follows.
    A natural generalization is the following.
	
	\begin{problem}
		Let $\mathcal{I}$ be an ideal of Borel sets on $\N^\N$. Is there a non-hyperfinite CBER $E$ such that for each $B \in \mc{I}$ the restriction $\restriction{E}{B}$ is hyperfinite? What if $\mc{I}$ is the ideal generated by compact sets?
	\end{problem}

	\bibliographystyle{alpha}
	\bibliography{ref(4)}

\end{document}